\documentclass[11pt]{amsart}
\usepackage[unicode,bookmarks,colorlinks]{hyperref}

\usepackage{amsmath,amssymb,amsthm}
\usepackage{amsfonts}
\usepackage[mathscr]{eucal}
\usepackage{mathrsfs}
\usepackage{graphicx}
\usepackage{color}
\usepackage{amsmath,amssymb,amsthm}
\usepackage{amsfonts}
\usepackage[mathscr]{eucal}
\newtheorem{theorem}{Theorem}[section]
\newtheorem{lemma}[theorem]{Lemma}

\theoremstyle{definition}

\theoremstyle{remark}
\newtheorem{remark}[theorem]{Remark}

\numberwithin{equation}{section}

\def\E{{\mathbb E}}
\def\P{{\mathbb P}}

\def\D{{\mathbb D}}

\begin{document}

\title{\bf Correlation structure of time-changed Pearson diffusions}



\author{Jebessa B. Mijena}
\address{Jebessa B. Mijena, 231 W. Hancock St, Campus Box 17, Department of Mathematics,
Georgia College \& State University, Milledgeville, GA 31061}
\email{jebessa.mijena@gcsu.edu}

\author{Erkan Nane}
\address{Erkan Nane, 221 Parker Hall, Department of Mathematics and Statistics,
Auburn University, Auburn, Al 36849}
\email{nane@auburn.edu}
\urladdr{http://www.auburn.edu/$\sim$ezn0001}

\begin{abstract}

The stochastic solution to  diffusion equations with polynomial coefficients is called a Pearson diffusion. If the  time derivative is replaced by a distributed fractional derivative, the stochastic solution is called a fractional Pearson diffusion. This paper develops a formula  for the covariance function of a fractional Pearson diffusion in steady state, in terms of generalized Mittag-Leffler functions. That formula shows that fractional Pearson diffusions are long-range dependent, with a correlation that falls off like a power law, whose exponent equals the smallest order of the distributed fractional derivative.
\end{abstract}


\maketitle


\section{Introduction}

In this paper we will study time--changed Pearson diffusions. Some versions of this process have been studied recently by Leonenko et al. \cite{leonenko-0, leonenko}. They considered the inverse stable subordinator as the time change process. They have studied the governing equations and correlation structure of the time changed Pearson diffusions. We will extend their results to  the time--changed Pearson diffusions where the time change  processes  are   inverse of mixtures of stable subordinators.

To introduce Pearson diffusion, let
\begin{equation}
\mu(x) = a_0 + a_1x\ \  \mathrm{and}\ \
D(x) =
\frac{\sigma^2(x)}{2}
= d_0 + d_1x + d_2x^2
. 
\end{equation}

The solution  $X_1(t)$ of the stochastic differential equation
\begin{equation}dX_1(t) =\mu(X_1(t))dt + \sigma (X_1(t))dW(t),
\end{equation}
where $W(t)$ is a standard Brownian motion, is
 called a Pearson diffusion.

 Special cases of this equation have been studied: $X_1(t)$ is called  Ornstein–-Uhlenbeck process \cite{uhlenbeck-30} when $\sigma(x)$ is a positive constant;
 $X_1(t)$  is  called the Cox–-Ingersoll-–Ross (CIR) process, when $d_2 = 0$, which is used in finance \cite{CIR-85}.

  The study of Pearson diffusions began with
Kolmogorov \cite{kolmogorov-31}. Let $p_1(x, t;y)$ denote the conditional probability density of $x = X_1(t)$
given $y = X_1(0)$, i.e., the transition density of this time-homogeneous Markov process. $p_1(x, t;y)$ is the fundamental solution to the
Kolmogorov backward  equation (Fokker–-Planck equation)
\begin{equation}\label{pearson-diffuion-density-pde}
\frac{\partial}{\partial t}p_1(x, t;y)=\mathcal{G}p_1(x, t;y) = \left[\mu(y)\frac{\partial}{\partial y} + \frac{\sigma^2(y)}{2}\frac{\partial^2}{\partial y^2}\right]p_1(x, t;y),
\end{equation}
with the  initial condition $p_1(x,0;y) = \delta(x-y)$. In this case  $X_1(t)$ is called  the stochastic solution to  the backward equation \eqref{pearson-diffuion-density-pde}.

The Caputo fractional derivative \cite{Caputo} is defined for $0<\beta<1$ as
\begin{equation}\label{CaputoDef}
\frac{\partial^\beta u(t,x)}{\partial t^\beta}=\frac{1}{\Gamma(1-\beta)}\int_0^t \frac{\partial
u(r,x)}{\partial r}\frac{dr}{(t-r)^\beta} .
\end{equation}
Its Laplace transform
\begin{equation}\label{CaputolT}
\int_0^\infty e^{-st} \frac{\partial^\beta u(t,x)}{\partial
t^\beta}\,dt=s^\beta \tilde u(s,x)-s^{\beta-1} u(0,x),
\end{equation}
where $\tilde u(s,x) = \int_0^\infty e^{-st}u(t,x)\, dt$ and incorporates the initial value in the same way as the first
derivative.
The distributed order fractional derivative is
\begin{equation}\label{DOFDdef}
\D^{(\mu)}u(t,x):=\int_0^1\frac{\partial^\beta u(t,x)}{\partial
t^\beta} \mu(d\beta),
\end{equation}
where $\mu$ is a finite Borel measure with $\mu(0,1)>0$.

The
solution to the distributed order fractional diffusion equation
\begin{equation}\label{DOFCPdef-eq}
\D^{(\mu)} u(t,y)=\int_0^1\frac{\partial^\beta u(t,y)}{\partial
t^\beta} \mu(d\beta)= \mathcal{G}u(t,y)=\left[\mu(y)\frac{\partial}{\partial y} + \frac{\sigma^2(y)}{2}\frac{\partial^2}{\partial y^2}\right]u(t,y),
\end{equation}
 is called distributed order(time-changed) Pearson diffusion.

A stochastic process $X(t), \ t>0$ with $\E(X(t))=0$ and $\E(X^2(t))<\infty$ is said to have short range dependence if for fixed $t>0$,  $\sum_{h=1}^\infty \E(X(t)X(t+h))<\infty$, otherwise it is said to have  long-range dependence.

Let $0<\beta_{1}<\beta_{2}<\cdot\cdot\cdot < \beta_{n}<1$.
In this paper we study the correlation structure of the stochastic solution of the distributed order time fractional equation \eqref{DOFCPdef-eq} with
$$\D^{(\mu)}u(t,y)=\sum_{i=1}^{n}c_{i}\frac{\partial^{\beta_{i}}u(t,y)}{\partial t^{\beta_{i}}},$$
and derive a formula \eqref{correlationfun} and \eqref{correlationfun2} for the correlation in terms of generalized Mittag-Leffler functions. In addition, we obtain an asymptotic expansion \eqref{asymptotic}, to show that the correlation falls off like $t^{-\beta_1}$ for large $t,$ thus demonstrating that distributed order fractional Pearson diffusions exhibit long-range dependence.

\section{Distributed order fractional Pearson Diffusion}

Let $m(x)$ be the steady-state distribution of $X_1(t)$. The generator associated with the backward equation \eqref{DOFCPdef-eq}
\begin{equation}\label{heat-generator}
\mathcal{G}p_1(x, t;y) = \left[\mu(y)\frac{\partial}{\partial y} + \frac{\sigma^2(y)}{2}\frac{\partial^2}{\partial y^2}\right]p_1(x, t;y),
\end{equation}
has a set of eigenfunctions   that solve the equation $\mathcal{G}Q_n(y) = -\lambda_nQ_n(y)$ with eigenvalues $0 = \lambda_0 < \lambda_1 < \lambda_2 < \cdots$ that form an orthonormal
basis for  $L^2(m(y)dy)$. In this paper we will consider three cases as in the papers \cite{leonenko-0, leonenko}:
\begin{enumerate}
\item  $d_1 = d_2 = 0$ and $d_0 > 0$, then $m(y)$ is a normal density, and $Q_n$ are Hermite polynomials;
\item
$d_2 = 0$,  $m(y)$ is a gamma density, and $Q_n$ are Laguerre polynomials;
\item   $D
''(y) < 0 $ with two positive real roots, $m(y)$ is
a beta density, and $Q_n$ are Jacobi polynomials.
\end{enumerate}
In the remaining cases, the spectrum of $\mathcal{G}$ has a continuous part, and some
moments of $X_1(t)$ do not exist. In every case, $m(y)$ is one of the Pearson distributions \cite{pearson-14}. For the remainder of this paper,
we will assume one of the three cases (Hermite, Laguerre, Jacobi), so that all moments exist.

By separation of variables we can show that the transition density of $X_1(t)$ is given by
\begin{equation}p_1(x, t;y) = m(x)
\sum_{n=0}^\infty
e^{-\lambda_n t}Q_n(x)Q_n(y). 
\end{equation}
See \cite{leonenko} and \cite[section 7.4]{meerschaert-skorskii-book} for more details of this derivation.

Let $ D(t)$ be a subordinator
with ${\mathbb E}[e^{-s D(t)}]=e^{-t\psi(s)}$, where
\begin{equation}\label{phiWdef}
\psi(s)=\int_0^\infty(e^{-s x}-1)\phi(dx) .
\end{equation}
Then the associated L\'evy measure is
\begin{equation}\label{psiWdef}
\phi(t,\infty)=\int_0^1 t^{-\beta}\nu(d\beta).
\end{equation}
An easy computation gives
\begin{equation}\begin{split}\label{psiW}
\psi(s)
&= \int_0^1  s^\beta \Gamma(1-\beta) \nu(d\beta)=\int_0^1  s^\beta \mu(d\beta) .
\end{split}\end{equation}
Here we define $\mu(d\beta)=\Gamma(1-\beta) \nu(d\beta)$.

Let
\begin{equation}\label{Epsi-def}
E(t)=\inf\{\tau\geq 0:\ D(\tau)>t\},
\end{equation}
be the inverse subordinator.
 Since
$\phi(0,\infty)=\infty$ in \eqref{phiWdef}, Theorem 3.1
 in \cite{mark} implies that $E(t)$ has a Lebesgue density
\begin{equation}\label{E-lebesgue-density}
f_{E(t)}(x)=\int_0^t \phi(t-y,\infty) P_{D(x)}(dy) .
\end{equation}
Note that $E(t)$ is almost surely continuous and  nondecreasing.

Since the  distributed order  time-fractional analogue \eqref{DOFCPdef-eq} to the  equation \eqref{pearson-diffuion-density-pde} is a fractional Cauchy problem of the form
\begin{equation}\label{DOFCPdef-1}
\D^{(\mu)} p_\mu(t,x)= \mathcal{G}p_\mu(t,x),
\end{equation}
a general semigroup result \cite[Theorem 4.2]{jebessa-nane-pams} implies that
\begin{equation}\label{distributed-transition-density}
p_\mu(x, t;y) =
\int_0^\infty
p_1(x,u;y)f_{E(t)}(u)du, 
\end{equation}
where
$f_{E(t)}(u)$ is the probability density of   inverse subordinator.
\begin{lemma}[\cite{m-n-v-jmaa}]\label{eigenvalue-problem}
For any $\lambda>0$, $h(t, \lambda)=\int_0^\infty e^{-\lambda
x}f_{E(t)}(x)\,dx=\E [e^{-\lambda E(t)}]$ is a mild solution of the
distributed-order fractional differential equation
\begin{equation}\label{dist-order-density-pde}
\D^{(\mu)}h(t,\lambda)=-\lambda h(t, \lambda); \ \ h(0, \lambda)=1.
\end{equation}

\end{lemma}

 Then it follows from Lemma \ref{eigenvalue-problem} and equation \eqref{distributed-transition-density}  that the transition density of $X_1(E(t))$ is given by
\begin{equation}\begin{split}
p_\mu(x, t;y)& =
m(x)
\sum_{n=0}^\infty Q_n(x)Q_n(y)
\int_0^\infty e^{-\lambda_n u} f_{E(t)}(u)du\\
&= m(x)
\sum_{n=0}^\infty Q_n(x)Q_n(y)
h(t,\lambda_n).
\end{split}
\end{equation}
We state the next theorem as a result  of  the observations above. This theorem extends the results in Leonenko et al. \cite{leonenko-0} and Meerschaert et al. \cite{m-n-v-jmaa}. The proof follows  with similar line of ideas as in  \cite{leonenko-0} and  \cite{m-n-v-jmaa}.
\begin{theorem}
Let $(l, L)$ be an interval such that $D(x) > 0$ for all $x \in  (l, L)$.
 Suppose that, the function $g \in  L^2(m(x)dx) $ is such that $\sum_{n}g_nQ_n$
with $g_n=
\int_l^L
g(x)Q_n(x)m(x)dx$ converges to $g$
uniformly on finite intervals $[y_1,y_2] \subset  (l, L)$. Then the fractional Cauchy problem
\begin{equation}\label{DOFCPdef-thm}
\D^{(\mu)} u(t,y)= \mathcal{G}u(t,y),
\end{equation}
with initial condition $u(0,y) = g(y)$ has a strong solution $u = u(t,y)$ given by
\begin{equation}\label{series-solution-frac-pde}
\begin{split}
u(t,y) &= \E(g(X_1(E(t)))|X_1(0)=y) =\int_l^L
p_\mu(x, t;y)g(x)dx
\\&
 =
m(x)
\sum_{n=0}^\infty g_n h(t,\lambda_n) Q_n(y)
\end{split}
\end{equation}
The series in \eqref{series-solution-frac-pde} converges absolutely for each fixed $t > 0,y \in (l, L)$, and \eqref{DOFCPdef-thm} holds pointwise.
\end{theorem}

Let $0<\beta_{1}<\beta_{2}<\cdot\cdot\cdot < \beta_{n}<1$.
In this paper we study the correlation structure of the stochastic solution of the distributed order time fractional equation \eqref{DOFCPdef-eq} with
$$\D^{(\mu)}=\sum_{i=1}^{n}c_{i}\frac{\partial^{\beta_{i}}g(x,t)}{\partial t^{\beta_{i}}},$$
this corresponds to the case where
\begin{equation}\label{n-term-laplace-exponent}
\psi(s)=c_{1}s^{\beta_{1}}+c_{2}s^{\beta_{2}}+\cdots +c_{n}s^{\beta_{n}}.
\end{equation}
In this case the L\'evy subordinator can be written as
$$D_\psi(t)=(c_1)^{1/\beta_1}D^1(t)+(c_2)^{1/\beta_2}D^2(t)+\cdots+ (c_n)^{1/\beta_n}D^n(t),$$
where $D^1(t), D^2(t),\cdots , D^n(t)$ are independent  stable subordinators of index $0<\beta_{1}<\beta_{2}<\cdot\cdot\cdot < \beta_{n}<1$.

In this paper we give the correlation function of the time-changed Pearson diffusion $X_1(E)$ where $E$ is the inverse of subordinator with Laplace exponent  given by \eqref{n-term-laplace-exponent}.

\section{Correlation structure}
In this section we present the correlation structure of the time-changed  Pearson diffusion $X_1(E(t))$  when $X_1(E(t))$ is the stochastic solution of the equation
$$c_1\frac{\partial^{\beta_1}p(x,t;y)}{\partial t^{\beta_1}}+c_2\frac{\partial^{\beta_2}p(x,t;y)}{\partial t^{\beta_2}}=\mathcal{G}p(x,t;y).$$
In this case  $E(t)$  is the inverse of $D(t)$  that has the following  Laplace exponent
\begin{equation}\label{psiD2}
\psi(s)= c_1s^{\beta_1} + c_2s^{\beta_2},
\end{equation}
for $c_1,c_2\geq 0$, $c_1 + c_2 = 1,$ and $\beta_1<\beta_2.$

In what follows we use notation for  the density of the inverse subordinator $E(t)$ as $f_{E(t)}(u)=f_t(u).$

Let $\Phi_\theta(t) = \int_{0}^\infty e^{-\theta u}f_t(u)\ du$. Using
$ \int^{\infty}_{0}e^{-st}f_{t}(u)dt=\frac{1}{s}\psi(s)e^{-u\psi(s)}$ (\cite{mark}, 3.13) and Fubini's theorem, the Laplace transform of $\Phi_\theta(t)$ is given by

\begin{eqnarray}\label{laplaceofPhi}
\mathcal{L}({\Phi}_\theta(t);s)&=& \int_{0}^{\infty}e^{-\theta u}\int_{0}^{\infty}e^{-st}f_t(u)dt\ du\nonumber\\
&=&\frac{\psi(s)}{s}\int_{0}^{\infty}e^{-u(\theta + \psi(s))}\ du\nonumber\\
&=&\frac{\psi(s)}{s(\theta + \psi(s))}= \frac{c_1s^{\beta_1-1} + c_2s^{\beta_2-1}}{\theta + c_1s^{\beta_1} + c_2s^{\beta_2}}.
\end{eqnarray}
In order to invert analytically the Laplace transform \eqref{laplaceofPhi}, we can apply the well-known expression of the Laplace transform of the generalized Mittag-Leffler function (see \cite{saxena}, eq. 9), i.e.

\begin{equation}\label{laplaceofGMF}
\mathcal{L}(t^{\gamma-1}E^{\delta}_{\beta,\gamma}(\omega t^{\beta});s)=s^{-\gamma}\left(1-\omega s^{-\beta}\right)^{-\delta},
\end{equation}
where $\mbox{Re}(\beta)>0, \mbox{Re}(\gamma)>0, \mbox{Re}(\delta)>0$ and $s>|\omega|^{\frac{1}{Re(\beta)}}.$ The Generalized Mittag-Leffler (GML) function is defined as
\begin{equation}\label{GMLfunction}
E^{\gamma}_{\alpha,\beta}(z) = \displaystyle\sum_{j=0}^{\infty}\frac{(\gamma)_jz^j}{j!\Gamma(\alpha j + \beta)},\ \ \ \alpha, \beta\in \mathbb{C}, Re(\alpha), Re(\beta), Re(\gamma)>0,
\end{equation}
where $(\gamma)_j=\gamma(\gamma + 1)\cdots (\gamma+j-1)$ (for $j=0,1,\ldots,\ \mbox{and}\ \gamma\neq 0$) is the Pochammer symbol and $(\gamma)_0=1.$ When $\gamma = 1$ \eqref{GMLfunction} reduces to the Mittag-Leffler function
\begin{equation}
E_{\alpha,\beta}(z) = \displaystyle\sum_{j=0}^{\infty}\frac{z^j}{\Gamma(\alpha j + \beta)}.
\end{equation}Now using formulae $(26)$ and $(27)$ of \cite{saxena}, we get
\begin{eqnarray}\label{thefunctPhi}
\Phi_\theta(t)
&=&\displaystyle\sum_{r=0}^{\infty}\left(-\frac{c_1t^{\beta_2-\beta_1}}{c_2}\right)^r E^{r+1}_{\beta_2, (\beta_2-\beta_1)r + 1}\left(-\frac{\theta t^{\beta_2}}{c_2}\right)\\&-&\displaystyle\sum_{r=0}^{\infty}\left(-\frac{c_1t^{\beta_2-\beta_1}}{c_2}\right)^{r+1} E^{r+1}_{\beta_2, (\beta_2-\beta_1)(r+1) + 1}\left(-\frac{\theta t^{\beta_2}}{c_2}\right).\nonumber
\end{eqnarray}
Clearly, $\Phi_\theta(0) = 1.$
\begin{remark}
Consider the special case $c_1=0,c_2=1:$ the formula for $\Phi_\theta(t)$ reduces, in this case, to
\begin{equation}
\Phi_\theta(t)= E_{\beta_2, 1}(-\theta t^{\beta_2})=\displaystyle\sum_{j=0}^{\infty}\frac{\left(-\theta t^{\beta_2}\right)^j}{\Gamma(1+j\beta_2)},
\end{equation}
as shown in Bingham \cite{bingham} and Bondesson, Kristiansen, and Steute \cite{bondesson} when $E(t)$ is standard inverse $\beta_2-$stable subordinator.
\end{remark}
Now we compute the expected value of the inverse subordinator $E(t)$. First, we find the Laplace transform of $\mathbb{E}(E(t)) = \int_{0}^{\infty}xf_t(x)\ dx$. Using Fubini's theorem, we have
\begin{eqnarray}\label{expectedvalueofE_t}
\mathcal{L}(\mathbb{E}(E(t));\lambda)&=&\int_{0}^{\infty}e^{-\lambda t}\mathbb{E}(E(t))\ dt \\
&=&\int_{0}^{\infty}x\int_{0}^{\infty}e^{-\lambda t}f_{t}(x)\ dt\ dx\nonumber\\
&=&\frac{\psi(\lambda)}{\lambda}\int_{0}^{\infty}x e^{-x\psi(\lambda)}\ dx=\frac{1}{\lambda\psi(\lambda)}\nonumber\\
&=&\frac{1}{c_1\lambda^{\beta_1 +1} + c_2\lambda^{\beta_2 + 1}} = \frac{1}{c_2}\frac{\lambda^{-(\beta_2 + 1)}}{1 + \frac{c_1}{c_2}\lambda^{-(\beta_2-\beta_1)}}.\nonumber
\end{eqnarray}
Therefore, using \eqref{laplaceofGMF} when $\delta = 1$ we have
\begin{equation}\label{expectedvalueexpre}
\mathbb{E}(E(t)) = \frac{1}{c_2}t^{\beta_2}E_{\beta_2-\beta_1, \beta_2 + 1}\left(-\frac{c_1}{c_2}t^{\beta_2-\beta_1}\right).
\end{equation}
\begin{remark}
Consider the special case $c_1=0,c_2=1:$ the expected value of $E(t)$ reduces, in this case, to
\begin{equation}
\mathbb{E}(E(t)) =\frac{t^{\beta_2}}{\Gamma(1+\beta_2)},
\end{equation}
thus giving the well-known formula $\mathbb{E}(E(t)) = t^{\beta_2}/\Gamma(1+\beta_2)$ for the mean of the standard inverse $\beta_2-$stable subordinator \cite[Eq.(9)]{baeumer}
\end{remark}
\subsection{Correlation function}
If the time-homogeneous Markov process $X_1(t)$ is in steady state, then its probability density ${ m}(x)$ stays the same over all time.  We will say that the time-changed Pearson diffusion $X_1(E(t))$ is in steady state if it starts with the distribution $m(x)$. The time-changed Pearson diffusion in steady state has mean $ \mathbb{E}[X_1(E(t))] =\mathbb{E}[X_1(t)]= m_1$ and variance Var$[X(E(t))] = $Var$[X_1(t)]=m_2^2$ which do not vary over time. The stationary Pearson diffusion has correlation function
\begin{equation}\label{pearsoncorrelation}
\mbox{corr}[X_1(t), X_1(s)] = \mbox{exp}(-\theta|t-s|),
\end{equation}
where the correlation parameter $\theta = \lambda_1$ is the smallest positive eigenvalue of the generator in equation \eqref{heat-generator}\cite{leonenko-0}. Thus the Pearson diffusion exhibits \textit{short-term dependence}, with a correlation function that falls off exponentially. The next result gives formula for the correlation function of time-changed Pearson diffusion in steady state. This is our main result.
\begin{theorem}\label{main-theorem}
Suppose that $X_1(t)$ is a Pearson diffusion in steady state, so that its correlation function is given by \eqref{pearsoncorrelation}. Then the correlation function of the corresponding time-changed Pearson diffusion $X(t) = X_1(E(t)),$ where $E(t)$ is an independent inverse  subordinator \eqref{Epsi-def} of $D(t)$ with Laplace exponent \eqref{psiD2}, is given by
\begin{equation}\label{correlationfun}
\mbox{corr}[X(t), X(s)] =\theta\int_{y=0}^{s}h(y)\Phi_\theta(t-y)\ dy + \Phi_\theta(t),
\end{equation}
where $h(y) = \frac{1}{c_2}y^{\beta_2 - 1}E_{\beta_2 - \beta_1, \beta_2}\left(-\frac{c_1}{c_2}y^{\beta_2-\beta_1}\right)$ and $\Phi_\theta(t)$ is given by \eqref{thefunctPhi}.
\end{theorem}

\begin{proof}[\bf  Proof of Theorem \ref{main-theorem}]
We use the method employed by Leonenko et al \cite{leonenko} with crucial changes.
Write \begin{eqnarray}\mbox{corr}[X(t), X(s)] &=& \mbox{corr}[X_1(E(t)), X_1(E(s))]\nonumber\\
&=&\int_{0}^{\infty}\int_{0}^{\infty}e^{-\theta|u-v|}H(du, dv),\label{corrfun}\end{eqnarray}
a Lebesgue-Stieltjes integral with respect to the bivariate distribution function $H(u, v) :=\mathbb{P}[E(t)\leq u, E(s)\leq v]$ of the process $E(t).$

To compute the integral in \eqref{corrfun}, we use the bivariate integration by parts formula \cite[Lemma 2.2]{gill}
\begin{eqnarray}
\int_0^a\int_0^b G(u,v)H(du, dv) &=& \int_0^a\int_0^b H([u, a]\times [v, b])G(du, dv)\nonumber\\ &+& \int_0^a H([u,a]\times(0,b])G(du,0)\nonumber\\ &+& \int_0^b H((0, a]\times [v,b])G(0, dv)\nonumber\\ &+& G(0,0)H((0, a]\times (0,b]),
\end{eqnarray}
with $G(u,v) = e^{-\theta|u-v|},$ and the limits of integration $a$ and $b$ are infinite:
\begin{eqnarray}
\int_0^\infty\int_0^\infty G(u,v)H(du, dv) &=& \int_0^\infty\int_0^\infty H([u, \infty]\times [v, \infty])G(du, dv)\nonumber\\ &+& \int_0^\infty H([u,\infty]\times(0,\infty])G(du,0)\nonumber\\ &+& \int_0^\infty H((0, \infty]\times [v,\infty])G(0, dv)\nonumber\\ &+& G(0,0)H((0, \infty]\times (0,\infty])\nonumber\end{eqnarray}
\begin{eqnarray}\label{corrfunction2}
&=&\int_0^\infty\int_0^\infty \mathbb{P}[E_t\geq u, E_s\geq v]G(du, dv)+\int_0^{\infty}\mathbb{P}[E_t\geq u]G(du,0)\nonumber\\ &+&\int_0^{\infty}\mathbb{P}[E_s\geq v]G(0,dv) + 1,\label{intbyparts}
\end{eqnarray}
since $E(t)>0$ with probability $1$ for all $t>0.$ Notice that $G(du,v) = g_v(u)\,du$ for all $v\geq 0,$ where
\begin{equation}\label{gfunction}
g_v(u) = -\theta e^{-\theta(u-v)}I\{u > v\} + \theta e^{-\theta(v-u)}I\{u\leq v\}.
\end{equation}
Integrate by parts to get
\begin{eqnarray}
\int_0^\infty\mathbb{P}[E(t)\geq u]G(du, 0)&=&\int_0^\infty (1-\mathbb{P}[E(t)<u])(-\theta e^{-\theta u})\,du\nonumber\\
&=&\left[e^{-\theta u}\mathbb{P}[E(t)\geq u]\right]_0^\infty + \int_0^\infty e^{-\theta u}f_t(u)\,du\nonumber\\
&=& \Phi_\theta(t) - 1.
\end{eqnarray}
Similarly, $$\int_0^{\infty}\mathbb{P}[E(s)\geq v]G(0, dv) = \int_0^\infty e^{-\theta v}f_s(v)\,dv -1=\Phi_\theta(s)-1,$$
and hence \eqref{corrfunction2} reduces to
\begin{equation}\label{genform}
\int_0^\infty\int_0^\infty G(u, v)H(du, dv) = I + \Phi_\theta(t) +  \Phi_\theta(s)- 1,
\end{equation}
where $$I = \int_0^\infty\int_0^\infty \mathbb{P}[E_t\geq u, E_s\geq v]G(du,dv).$$
Assume (without loss of generality) that $t\geq s$. Then $E_t\geq E_s$, so $\mathbb{P}[E(t)\geq u, E(s)\geq v] = \mathbb{P}[E(s)\geq v]$ for $u\leq v.$ Write $I = I_1 + I_2 + I_3,$ where
\begin{eqnarray}
I_1:= \int_{u < v}\mathbb{P}[E(t)\geq u, E(s)\geq v]G(du,dv) = \int_{u<v}\mathbb{P}[E(s)\geq v]G(du,dv)\nonumber\\ I_2:= \int_{u = v}\mathbb{P}[E(t)\geq u, E(s)\geq v]G(du,dv) = \int_{u=v}\mathbb{P}[E(s)\geq v]G(du,dv)\nonumber\\ I_3:= \int_{u \geq v}\mathbb{P}[E(t)\geq u, E(s)\geq v]G(du,dv).\nonumber \ \  \ \ \ \ \ \  \ \ \ \ \ \  \ \ \ \ \ \ \ \ \ \ \ \ \ \ \ \  \ \ \ \  \ \nonumber
\end{eqnarray}
Since $G(du,dv) = -\theta^2 e^{-\theta(v-u)}\,du\,dv$ for $u < v$, we may write
\begin{eqnarray}
I_1 &=& -\theta^2\int_{v=0}^\infty\int_{u=0}^v \mathbb{P}[E(s)\geq v]e^{\theta(u-v)}\,du\,dv\nonumber\\ &=&-\theta\int_{v=0}^\infty \mathbb{P}[E(s)\geq v]\left(1 - e^{-\theta v}\right)\,dv\nonumber\\ &=& -\theta\int_{v=0}^\infty \mathbb{P}[E(s)\geq v]\,dv + \theta\int_0^\infty e^{-\theta v}\mathbb{P}[E_s \geq v]\,dv\nonumber\\ &=& -\theta\mathbb{E}[E(s)] + \theta\int_0^\infty e^{-\theta v}\mathbb{P}[E(s) \geq v]\,dv,\nonumber
\end{eqnarray}
using the well-known formula $\mathbb{E}[X] = \int_0^\infty\mathbb{P}[X\geq x]\,dx$ for any positive random variable. Using integration by parts
$$\int_0^\infty e^{-\theta v}\mathbb{P}[E(s) \geq v]\,dv = \frac{1}{\theta} - \frac{1}{\theta}\int_0^\infty e^{-\theta v}f_s(v)\,dv=\frac{1}{\theta} - \frac{\Phi_\theta(s)}{\theta}.$$
So,
\begin{equation}\label{part_one}
I_1 = -\theta\ \mathbb{E}[E(s)] - \Phi_\theta(s) + 1.
\end{equation}
Since $G(du,v) = g_v(u)du,$ where the function \eqref{gfunction} has jump of size $2\theta$ at the point $u = v$, we also have
\begin{equation}
I_2 = 2\theta\int_0^\infty\mathbb{P}[E(s)\geq v]\,dv = 2\theta\ \mathbb{E}[E(s)].
\end{equation}
Since $G(du,dv) = -\theta^2e^{-\theta(u-v)}\,du\ dv$ for $u>v$ as well, we have
\begin{eqnarray}\label{part3}
I_3 = -\theta^2\int_{v=0}^\infty\mathbb{P}[E(t)\geq u, E(s)\geq v]\int_{u=v}^\infty e^{-\theta(u-v)}\, du\, dv.
\end{eqnarray}
Next, we obtain an expression for $\mathbb{P}[E(t)\geq u, E(s)\geq v].$ Since the process $E(t)$ is inverse to the stable subordinator $D(u),$ we have $\{E(t)> u\}=\{D(u)<t\}$ \cite[Eq. (3.2)]{mark2}, and since $E(t)$ has a density, it follows that $\mathbb{P}[E(t)\geq u, E(s)\geq v] = \mathbb{P}[D(u)<t, D(v)<s].$ Since $D(u)$ has stationary independent increments, it follows that
\begin{eqnarray}
\P[E(t)\geq u, E(s) \geq v]& =& \mathbb{P}[D(u)<t, D(v)<s]\nonumber\\
&=& \mathbb{P}[(D(u) - D(v)) + D(v) < t, D(v) <s]\nonumber\\
&=&\int_{y=0}^sg(y,v)\int_{x=0}^{t-y}g(x, u -v)\, dx\,dy,
\end{eqnarray}
substituting the above expression into \eqref{part3} and using the Fubini Theorem, it follow that
\begin{eqnarray}\label{part3estimate}
I_3&=& -\theta^2\int_{y=0}^s\int_{x=0}^{t-y}\int_{v=0}^\infty g(y,v)\ dv\int_{u=v}^\infty g(x,u-v)e^{-\theta(u-v)}du\ dx\ dy\nonumber\\
&=&-\theta^2\int_{y=0}^s\int_{x=0}^{t-y}\int_{v=0}^\infty g(y,v)\ dv\int_{z=0}^\infty g(x,z)e^{-\theta z}dz\,dx\,dy.
\end{eqnarray}
Let $h(y) = \int_{v=0}^\infty g(y,v)\ dv$ and $k(\theta, x) = \int_{z=0}^\infty g(x,z)e^{-\theta z}\ dz.$ So, the Laplace transform of $h(y)$ is given by
\begin{eqnarray}
\mathcal{L}(h(y);s) = \int_{y=0}^\infty e^{-sy}h(y)\ dy &= & \int_{y=0}^\infty e^{-sy}\int_{v=0}^\infty g(y,v)\,dv\,dy\nonumber\\
&=&\int_{y=0}^\infty \int_{v=0}^\infty e^{-sy}g(y,v)\,dy\,dv\nonumber\\
&=&\int_{v=0}^\infty e^{-v\psi(s)}\ dv\nonumber\\
&=&\frac{1}{\psi_(s)} = \frac{1}{c_1s^{\beta_1}+c_2s^{\beta_2}}.\nonumber
\end{eqnarray}
Now to get $h(y)$ take inverse Laplace of $\mathcal{L}(h(y);s)$ by \eqref{laplaceofGMF} and $\delta=1$. This implies,
\begin{equation}\label{gestimate}h(y) = \frac{1}{c_2}y^{\beta_2 - 1}E_{\beta_2 - \beta_1, \beta_2}\left(-\frac{c_1}{c_2}y^{\beta_2-\beta_1}\right).\end{equation}
Similarly, take the Laplace transform of $k(\theta, x):$
\begin{eqnarray}\label{estimateofkfun}
\mathcal{L}(k(\theta, x);s) &=& \int_{x=0}^\infty e^{-sx}\int_{z=0}^\infty e^{-\theta z}g(x,z)\,dz\,dx\nonumber\\
&=&\int_{z=0}^\infty e^{-\theta z}\int_{x=0}^\infty e^{-s x}g(x,z)\,dx\,dz\nonumber\\
&=&\int_{z=0}^\infty e^{-\theta z} e^{-z\psi(s)}dz\nonumber\\
&=& \frac{1}{\theta + \psi(s)}\nonumber\\
& =& \frac{1}{\theta + c_1s^{\beta_1}+c_2s^{\beta_2}} = \frac{1}{\theta}\left(1-\frac{c_1s^{\beta_1} + c_2s^{\beta_2}}{\theta + c_1s^{\beta_1} + c_2s^{\beta_2}}\right)=\mathcal{L}\left(-\frac{1}{\theta}\frac{d}{dx}\Phi_\theta(x)\right).
\end{eqnarray}
The uniqueness theorem of the Laplace transform applied to the $x-$variable implies that for any $\theta>0$ we have
\begin{equation}\label{relationbnkandPhi}
k(\theta,x) = -\frac{1}{\theta}\frac{d}{dx}\Phi_\theta(x).
\end{equation}
Substituting \eqref{gestimate} and \eqref{relationbnkandPhi} in \eqref{part3estimate}, we have
\begin{eqnarray}\label{I3estimate}
I_3 &=& -\theta^2\int_{y=0}^s\int_{x=0}^{t-y}h(y)\left(-\frac{1}{\theta}\frac{d}{dx}\Phi_\theta(x)\right)dx\ dy\nonumber\\
&=&\theta\int_{y=0}^s h(y)\int_{x=0}^{t-y}\frac{d}{dx}\Phi_\theta(x)\ dx\ dy\nonumber\\
&=&\theta\int_{y=0}^{s}h(y)(\Phi_\theta(t-y)-1)\ dy\nonumber\\
&=&\theta\int_{y=0}^{s}h(y)\Phi_\theta(t-y)\ dy - \theta\int_{y=0}^{s}h(y)\ dy.
\end{eqnarray}
Using properties of Laplace transform, we have
$$\mathcal{L}\left(\int_{y=0}^{s}h(y)\ dy;\mu\right)=\frac{1}{\mu}\mathcal{L}(h(y);\mu)=\frac{1}{c_1\mu^{\beta_1+1}+c_2\mu^{\beta_2+1}}.$$
Hence, using uniqueness of Laplace transform
$$\int_{y=0}^{s}h(y)\ dy = \mathbb{E}(E(s)).$$
Therefore,
\begin{eqnarray}\label{I3estimate}
I_3 = \theta\int_{y=0}^{s}h(y)\Phi_\theta(t-y)\ dy - \theta\mathbb{E}(E(s)).\end{eqnarray}

Now it follows using \eqref{corrfun} and \eqref{genform} that
\begin{eqnarray}
\mbox{corr}[X(t), X(s)]
&=&\int_{0}^{\infty}\int_{0}^{\infty}e^{-\theta|u-v|}H(du, dv)\nonumber\\
&=&I_1 + I_2 + I_3 +  \Phi_\theta(t) + \Phi_\theta(s) - 1\nonumber\\
&=&\left[-\theta\mathbb{E}(E(s)) -\Phi_\theta(s) + 1\right] + 2\theta\mathbb{E}(E(s))\nonumber\\
&+&\theta\int_{y=0}^{s}h(y)\Phi(t-y)\ dy - \theta\mathbb{E}(E(s))+\Phi_\theta(t) + \Phi_\theta(s) - 1\nonumber\\
&=&\theta\int_{y=0}^{s}h(y)\Phi_\theta(t-y)\ dy + \Phi_\theta(t),\nonumber
\end{eqnarray}
which agrees with \eqref{correlationfun}.
\end{proof}

\begin{remark}
When $t=s,$ it must be true that $\mbox{corr}[X(t), X(s)]=1$. To see that this follows from \eqref{correlationfun}, recall the formula for Laplace transform of the convolution:
$$\mathcal{L}((h *\Phi_\theta)(t)) = \mathcal{L}(h(t))\mathcal{L}(\Phi_\theta(t)).$$
So,
\begin{eqnarray}
\mathcal{L}\left(\int_{y=0}^{t}h(y)\Phi_\theta(t-y)\ dy;\mu\right)&=&\mathcal{L}\left(h(t);\mu\right)\mathcal{L}\left(\Phi_\theta(t);\mu\right)\nonumber\\
&=&\left(\frac{1}{c_1s^{\beta_1}+c_2s^{\beta_2}}\right)\left(\frac{c_1s^{\beta_1-1}+c_2s^{\beta_2-1}}{\theta + c_1s^{\beta_1}+c_2s^{\beta_2}}\right)\nonumber\\
&=&\frac{s^{-1}}{\theta + c_1s^{\beta_1}+c_2s^{\beta_2}}.\nonumber
\end{eqnarray}
Using properties of Laplace transform and \eqref{estimateofkfun}, we see that
$$\int_{y=0}^{t}h(y)\Phi_\theta(t-y)\ dy = \int_{0}^{t}-\frac{1}{\theta}\frac{d}{dy}\Phi_\theta(y)\ dy=\frac{1 - \Phi_\theta(t)}{\theta}.$$
Then, it follows from \eqref{correlationfun} that $\mbox{corr}[X(t), X(s)]=1$ when $t=s.$
\end{remark}

\begin{remark}
Recall \cite[Eq. 2.59, p.59]{beghin} that
\begin{equation}\label{asympoticofGML}
E^{k}_{v,\beta}(-ct^v)=\frac{1}{c^kt^{vk}\Gamma(\beta-vk)} + o(t^{-vk}),\ \ \ t\rightarrow \infty,
\end{equation}
\begin{equation}\label{asymptoticforsmalltime}
E^{k}_{v,\beta}(-ct^v)\simeq \frac{1}{\Gamma(\beta)}-\frac{ct^{v}k}{\Gamma(\beta+v)},\ \ \ \ \ 0<t<<1.
\end{equation}

For $k=1$, using \eqref{asympoticofGML} the asymptotic behavior of $\mathbb{E}(E(t))$ for $t\rightarrow\infty$ is given by

$$
\E(E(t))=\frac{t^{\beta_1}}{c_1\Gamma(1+\beta_1)}+o(t^{\beta_1 - \beta_2}),\ \ \ t\to\infty.
$$
Similarly, the asymptotic behavior for small $t$ can be deduced using \eqref{asymptoticforsmalltime} when $k=1$:
$$\mathbb{E}(E(t))\simeq\frac{t^{\beta_2}}{c_2\Gamma(1+\beta_2)}-\frac{c_1t^{2\beta_2-\beta_1}}{c_2\Gamma(1+2\beta_2-\beta_1)},\ \ \ 0<t<<1.$$
\end{remark}
\begin{remark}\label{asympoticforPhi}
Stationary Pearson diffusion exhibit short-range dependence, since their correlation function \eqref{pearsoncorrelation} falls off exponentially fast. However, the correlation function of time-changed Pearson diffusion falls off like a power law with exponent $\beta_1\in(0,1), (\beta_1 < \beta_2)$ and so this process exhibits long-range dependence. To see this, fix $s>0$ and recall 
that by \eqref{relationbnkandPhi}
$$\Phi_\theta(t) = 1-\int_{x=0}^{t}\theta k(\theta, x)\,dx.$$
For fixed $\theta$, $\theta k(\theta, x)$ is a density function for $x\geq 0.$ Since $$\lim_{s\rightarrow 0}\frac{\theta}{c_1}s^{-\beta_1}\left(\frac{c_{1}s^{\beta_{1}}+c_{2}s^{\beta_{2}}}{\theta+c_{1}s^{\beta_{1}}+c_{2}s^{\beta_{2}}}\right)= 1.$$
Then by \cite[Example. (c), p.447]{feller} we get

$$\Phi_\theta(t)\simeq \frac{c_1}{\Gamma(1-\beta_1)\theta t^{\beta_1}}, \ \ \ \ \ t\rightarrow \infty,$$
which depends only on the smaller fractional index $\beta_1.$ You can also see \cite[Eq. 2.64]{beghin}.
Then $$\Phi_\theta(t(1-sy/t))\simeq \frac{c_1}{\Gamma(1-\beta_1)\theta t^{\beta_1}(1-sy/t)^{\beta_1}}, \ \ \ t\rightarrow\infty\ \ \mbox{for any}\ \ y\in[0,1].$$
Using dominated convergence theorem $(|\Phi_\theta(t)|\leq 1)$ and \cite[Eq.(1.99)]{podlubny} we get
\begin{eqnarray}
\theta\int_{y=0}^{s}h(y)\Phi_\theta(t-y)\ dy&=&\theta s\int_{0}^{1}h(sz)\Phi_\theta(t(1-sz/t))\ dz\nonumber\\
&\sim& \frac{sc_1}{\Gamma(1-\beta_1)t^{\beta_1}}\int_{0}^{1}h(sz)\ dz\nonumber\\
&=&\frac{c_1}{c_2}\frac{s^{\beta_2}E_{\beta_2-\beta_1,\beta_2+1}\left(-\frac{c_1}{c_2}s^{\beta_2-\beta_1}\right)}{t^{\beta_1}\Gamma(1-\beta_1)},\nonumber
\end{eqnarray}
as $t\rightarrow\infty.$ It follows from \eqref{correlationfun} that for any fixed $s>0$ we have
\begin{equation}\label{asymptotic}
\mbox{corr}[X(t), X(s)] \sim \frac{1}{t^{\beta_1}\Gamma(1-\beta_1)}\left(\frac{c_1}{\theta}+\frac{c_1}{c_2}s^{\beta_2}E_{\beta_2-\beta_1,\beta_2+1}\left(-\frac{c_1}{c_2}s^{\beta_2-\beta_1}\right)\right),\ \
\end{equation}
as $t\rightarrow\infty.$

Now if we also let $s\rightarrow\infty,$ using \eqref{asympoticofGML} when $k=1:$
\begin{equation}
\mbox{corr}[X(t), X(s)] \sim \frac{1}{t^{\beta_1}\Gamma(1-\beta_1)}\left(\frac{c_1}{\theta}+\frac{s^{\beta_1}}{\Gamma(1+\beta_1)}\right),
\end{equation}
as $t\rightarrow\infty$ and $s\rightarrow\infty.$
\end{remark}

With  careful changes to the proof of Theorem \ref{main-theorem} we can prove the following extension.
\begin{theorem}\label{main-theorem2}
Suppose that $X_1(t)$ is a Pearson diffusion in steady state, so that its correlation function is given by \eqref{pearsoncorrelation}. Then the correlation function of the corresponding time-changed Pearson diffusion $X(t) = X_1(E(t)),$ where $E(t)$ is an independent inverse  subordinator \eqref{Epsi-def} of $D(t)$ with Laplace exponent \eqref{n-term-laplace-exponent}, is given by
\begin{equation}\label{correlationfun2}
\mbox{corr}[X(t), X(s)] =\theta\int_{y=0}^{s}h_n(y)\Phi_{\theta,n}(t-y)\ dy + \Phi_{\theta,n}(t),
\end{equation}
where Laplace transform of $h_n$ is given by
$$\tilde{h}_n(s)=\frac{1}{\psi(s)}=\frac{1}{c_{1}s^{\beta_{1}}+c_{2}s^{\beta_{2}}+\cdots +c_{n}s^{\beta_{n}}},$$
and the Laplace transform of  $\Phi_{\theta,n}$ is given by
$$
\tilde{\Phi}_{\theta,n}(s)=\frac{\psi(s)}{s(\theta+\psi(s))}=\frac{c_{1}s^{\beta_{1}}+c_{2}s^{\beta_{2}}+\cdots +c_{n}s^{\beta_{n}}}{s(\theta+c_{1}s^{\beta_{1}}+c_{2}s^{\beta_{2}}+\cdots +c_{n}s^{\beta_{n}})}.
$$
In this case
$\Phi_{\theta,n}(t)=\E(e^{-\theta E(t)})$ is the laplace transform of the inverse subordinator $E(t)$.

\end{theorem}

\begin{remark}
Using \cite[Eq.(5.37)]{podlubny} and \eqref{laplaceofGMF} we get the expression for $h_n(y):$
\begin{eqnarray}
h_n(y) &=& \frac{1}{c_n}\displaystyle\sum_{m=0}^{\infty}(-1)^m\displaystyle\sum_{\substack{k_0+k_1+\cdots+k_{n-2}=m\\
k_0\geq0, k_1\geq 0,\ldots, k_{n-2}\geq 0
}}\left(m;k_0, k_1, \ldots, k_{n-2}\right)\\
&&\prod_{i=0}^{n-2}\left(\frac{c_i}{c_n}\right)^{k_i}y^{(\beta_n-\beta_{n-1})(m+1)+\beta_{n-1}+\sum_{i=0}^{n-2}(\beta_{n-1}-\beta_i)k_i - 1}\nonumber\\
&\times&E^{m+1}_{\beta_n-\beta_{n-1}, (\beta_n-\beta_{n-1})(m+1)+\beta_{n-1}+\sum_{i=0}^{n-2}(\beta_{n-1}-\beta_i)k_i }\left(\frac{-c_{n-1}}{c_n}y^{\beta_n-\beta_{n-1}}\right),\nonumber
\end{eqnarray}
where $(m;k_0, k_1,\ldots,k_{n-2})$ are the multinomial coefficients.
\end{remark}
\begin{remark}
As in remark \eqref{asympoticforPhi}
$$\Phi_{\theta,n}(t) = 1-\int_{x=0}^{t}\theta k(x,\theta)\,dx,$$
where in this case $k(x,\theta)$ corresponds to the general case.
Since $$\lim_{s\rightarrow 0}\frac{\theta}{c_1}s^{-\beta_1}\left(\frac{c_{1}s^{\beta_{1}}+c_{2}s^{\beta_{2}}+\cdots +c_{n}s^{\beta_{n}}}{\theta+c_{1}s^{\beta_{1}}+c_{2}s^{\beta_{2}}+\cdots +c_{n}s^{\beta_{n}}}\right)= 1.$$
Then by \cite[Example. (c), p.447]{feller} we have
$$\Phi_{\theta,n}(t) \sim \frac{c_1}{\theta\Gamma(1-\beta_1)t^{\beta_1}},\ \ \ \mbox{as}\ \ \ t\rightarrow\infty,$$
which depends only on the smaller fractional index $\beta_1.$
Hence, for fixed $s>0$
\begin{equation}
\mbox{corr}[X(t), X(s)]\sim \frac{1}{t^{\beta_1}\Gamma(1-\beta_1)}\left(\frac{c_1}{\theta}+c_1\int_{0}^{s}h_n(y)\,dy\right).
\end{equation}

Therefore in this case the time-changed Pearson diffusion exhibits long-range dependence.
\end{remark}


\begin{thebibliography}{99}

\bibitem{baeumer} B. Baeumer, M.M. Meerschaert, {Fractional diffusion with two time scales,} Physica A 373 (2007) 237-251.
\bibitem{beghin} L.Beghin. {Random-time processes governed by differential equations of fractional distributed order}. Probab. Theory and Rel. Fields. {\bf 142} (2008), no. 3-4,  313-338.
\bibitem{bingham} N.H. Bingham, Limit theorems for occupation times of Markov processes, Z. Wahrscheinlichkeitstheor. Verwandte Geb. 17 (1971) 1--22.

\bibitem{bondesson} L. Bondesson, G. Kristiansen, F. Steutel, Infinite divisibility of random variables and their integer parts, Statist. Probab. Lett. 28 (1996) 271--278.
\bibitem{Caputo}  M. Caputo, Linear models of dissipation whose Q is almost frequency independent,
  Part II. {\it Geophys. J. R. Astr. Soc.} {\bf 13} (1967) 529--539.

\bibitem{CIR-85}J.C. Cox, J.E. Ingersoll Jr., S.A. Ross, A theory of the term structure of interest rates, Econometrica 53 (1985) 385–407.

\bibitem{feller} W. Feller, {An introduction to probability theory and its applications}, Volume II.

\bibitem{gill} R.D. Gill, M.J. van der Laan, J.A. Wellner, {Inefficient estimators of the bivariate survival function for three models,} Ann. Inst. Henri Poincar$\acute{\mbox{e}}$ 31 (3) (1995) 545 - 597.

\bibitem{kolmogorov-31}A.N. Kolmogorov, \"Uber die analytischen Methoden in der Wahrscheinlichkeitsrechnung (on analytical methods in probability theory), Math. Ann. 104
(1931) 415–458.
\bibitem{leonenko-0} N.N. Leonenko, M.M. Meerschaert, and A. Sikorskii, Fractional Pearson diffusions, Journal of Mathematical Analysis and Applications, Vol. 403 (2013), No. 2, pp. 532–546.

\bibitem{leonenko}N.N. Leonenko, M.M. Meerschaert, and A. Sikorskii, Correlation Structure of Fractional Pearson Diffusions, Computers and Mathematics with Applications, Vol. 66 (2013), No. 5, pp. 737–745.
\bibitem{mark2} M. M. Meerschaert, H.P. Scheffler, Limit theorems for continuous time random walks with infinite mean waiting time, J. Appl. Probab. 41 (2004) 623 - 638

\bibitem{mark} M. M. Meerschaert, H.P. Scheffler (2008), Triangular array limits for continuous time random walks, Stochastic processes and their applications, 1606-1633.


\bibitem{m-n-v-jmaa} Meerschaert, M. M., Nane, E. and Vellaisamy, P. (2011). Distributed-order fractional diffusions on bounded domains. J. Math. Anal. Appl. {\bf 379} (2011) 216-228.

\bibitem{meerschaert-skorskii-book} M.M. Meerschaert and A. Sikorskii, Stochastic Models for Fractional Calculus, De Gruyter Studies in Mathematics Vol. 43, 2012,
\bibitem{jebessa-nane-pams} J. Mijena and E. Nane. Strong analytic solutions of fractional Cauchy problems.  (To appear)  Proceedings of the American Mathematical Society. Url:http://arxiv.org/abs/1110.4158

 \bibitem{pearson-14}K. Pearson, Tables for Statisticians and Biometricians, Cambridge University Press, Cambridge, UK, 1914

\bibitem{podlubny} I. Podlubny, {Fractional differential equations}, Mathematics in Science and Engineering, Volume 198.

\bibitem{saxena} R.K. Saxena, A.M. Mathai, H.J. Haubold. {Reaction-diffusion  systems and nonlinear waves, Astrophysics and Space Science,} {\bf 305}, 297-303

\bibitem{uhlenbeck-30}G.E. Uhlenbeck, L.S. Ornstein, On the theory of Brownian motion, Phys. Rev. 36 (1930) 823–841





\end{thebibliography}
\end{document}